\documentclass{ifacconf}

\usepackage{graphicx}      
\usepackage{natbib}        
\usepackage{amsmath}
\usepackage[caption=false,font=footnotesize]{subfig} 
\usepackage{stfloats}   
\usepackage{amssymb}
\usepackage{xcolor}
\usepackage{algorithm}
\usepackage[noend]{algpseudocode}
\algrenewcommand\algorithmicrequire{\textbf{Input:}}
\algrenewcommand\algorithmicensure{\textbf{Output:}}

\begin{document}
	\begin{frontmatter}
		
		\title{Pursuing Optimal Stepsize in Adaptive Gradient-Based Quadratic Optimization}
		
		\thanks[footnoteinfo]{{This work was supported in part by China Scholarship Council, and partially by National Major Science and Technology Project of China under Grant No. TC240A9ED-74.}}
		
		\author[se-uni]{Yifan Wang}
        \author[unipd]{Luca Ballotta}
		\author[unipd]{Ruggero Carli}
        \author[se-uni]{Xianghui Cao}
		\author[unipd]{Luca Schenato}
		
		\address[se-uni]{School of Automation, Southeast University, Nanjing, China \\
        (e-mail: \{evan, xhcao\}@seu.edu.cn)}
        \address[unipd]{Department of Information Engineering, University of Padova,
        Padova, Italy
        (e-mail: \{ballotta, carlirug, schenato\}@dei.unipd.it)}
		
		\begin{abstract}
			In this paper, we address the problem of achieving fast convergence in gradient descent for quadratic functions without relying on \emph{a-priori} knowledge of global function parameters.
			Inspired by adaptive stepsize algorithms for smooth convex functions,
			we propose a computationally lightweight strategy based on running estimates of minimal and maximal local curvatures.
			We prove that our proposed algorithm converges to the optimal constant stepsize which achieves the fastest convergence.
			Simulations show that the convergence rate achieved by our proposed algorithm is comparable or superior to recent adaptive approaches both in the quadratic case under consideration and in a preliminary test on logistic classification.
		\end{abstract}
		
		
	\end{frontmatter}
	
	\section{Introduction}
	Convex optimization is the core technology for solving many control, scheduling, and learning problems \citep{mpc,evan,ruggero}.
	Convergence speed and stability of optimization algorithms are critical performance requirements, especially in online optimization.
	It is well known that the stepsize of an algorithm is a key parameter affecting both convergence rate and stability.
	However, effectively choosing the stepsize typically requires tedious tuning, even for standard algorithms such as gradient descent.
	In fact, many existing optimization algorithms use a constant stepsize whose selection relies on the knowledge of a global Lipschitz constant, which is difficult to obtain in practice and often leads to overly conservative choices and slow convergence.
	On the other hand, methods that use decaying stepsizes inevitably cause slow convergence.
	
	Recently, the analysis and design of adaptive stepsizes for gradient descent have become a very active research topic to go beyond those performance limitations. For instance, \cite{linesearch} proposed line-search or backtracking strategies to adapt the stepsize, which require additional gradient evaluations and thus a substantial computational burden. AdaGrad-Norm stepsizes by \cite{adagrad} adaptively modifies the proximal function, which significantly simplifies setting a learning rate. Another well-known method is Polyak's stepsize by \cite{polyak}.
	However,
	AdaGrad-type and Polyak's rules require prior knowledge of the distance to a solution or the optimal value of cost function, which are unrealistic to precisely evaluate.
    Subsequent works, e.g., \cite{adagrad_ICML_2023} are devoted to circumventing this problem, but computing the stepsize still involves expensive operations such as multiple evaluations of the objective function and its gradient or high-dimensional running statistics updated at every iteration.
	Therefore, they are not considered truly adaptive methods.
	
	The optimization literature has explored adaptive stepsizes that adjust to the local smoothness of the objective.
	An early work in this direction is Barzilai-Borwein (BB) method by \cite{bb}, which adaptively computes stepsize using two consecutive gradients. However, BB method can diverge, in general. 
	\cite{adaptive_without_descent} proposed adaptive gradient descent (AdGD) for smooth convex functions and established its convergence with a novel proof.
	AdGD was further developed to speed up the convergence rate with a larger stepsize bound by \cite{adaptive_proximal,adgd_proximal2}. 
	\cite{Adabb} proposed a variant of the BB method that globally converges for general convex function, by incorporating a rate control term for the stepsize inspired by AdGD.
	\cite{iannelli2025} made the first attempt to combine past and one-step ahead information to estimate the local curvature.
	In light of these recent developments,
	we use estimation of local curvature to develop our adaptive algorithm.
	
    Although the aforementioned works focus on adaptive
    stepsizes to overcome conservative global Lipschitz con-
    stant and thus speed up convergence, they overlook that
    adaptively seizing optimal stepsizes is possible to further improve
    the convergence rate. 
	Inspired by existing adaptive rules, we make the first attempt to adaptively compute optimal constant stepsize in terms of the convergence rate.
	Note that such optimal constant stepsizes cannot be generally computed \emph{a priori} and thus require adaptive strategies.
	We focus on quadratic cost functions because we can obtain closed-form expressions of convergence rate and achieve analytical insight.
	These findings will be useful to future efforts on convex functions where exact quantification of convergence rate and knowledge of global parameters such as Lipschitz constant is unrealistic.
	
	\subsubsection*{Organization:} We present the problem statement in Section~\ref{sec:problem-statement}.
	In Section~\ref{sec:preliminaries}, we introduce preliminaries on estimating local curvature under constant and varying stepsizes. We give the proposed adaptive stepsize law and analyze its properties in Section~\ref{sec:main-results}. In Section~\ref{sec:numerical-study}, we show distinct features of the proposed stepsize law by numerical tests and compare it with recent strategies.
	Finally, we come to the conclusion in Section~\ref{sec:conclusion}. 
	
	\subsubsection*{Notation:} We use $\mathbb{R}^n$ for $n$-dimensional real vector space, and $\mathbb{R}^{n\times m}$ to denote the set of $n\times m$ real matrices. $\mathbb{N}_+$ denotes the set of positive natural numbers.
	
	\section{Problem Statement}\label{sec:problem-statement}
	Let us begin with the convex optimization problem 
	\begin{equation}\label{problem0}
		\min_{x\in\mathbb{R}^n} f(x).
	\end{equation}
	The cost function $f(x)$ is $L$-smooth and $\mu$-strongly convex.
	To solve problem \eqref{problem0}, one of the basic methods is gradient descent (GD) that recursively generates the iterates
	\begin{equation}\label{gd0}
		x_{k+1}=x_k-\alpha_k \nabla f(x_k),
	\end{equation}
	where $\alpha_k$ is the stepsize (possibly constant $\alpha_{k}\equiv\alpha$).
	
	To characterize the effectiveness of our algorithm,
	it is instrumental to define Lyapunov convergence, a metric that characterizes the rate (speed) of a time-varying system. 
	\begin{defn}\label{def1}
		[\cite{lyapunov_exponent}] The Lyapunov convergence of a dynamical system with error state $e_k:=x_k-x^*$ is
		\begin{equation}
			\tilde{\lambda}:=\lim_{k\rightarrow {+\infty}}\sup{\left( \frac{\ln{\|e_k\|}}{k} \right)},
		\end{equation}
        where $x^*$ is an optimal solution.
	\end{defn}
	Definition 1 implies $\tilde{\lambda}>0$ if $\lim_{k\rightarrow {+\infty}}\|e_k\|= {+\infty}$ and $\tilde{\lambda}<0$ if $\lim_{k\rightarrow{+\infty}}\|e_k\|=0$. We denote $\rho:=\exp(\tilde{\lambda})$ which represents the (Lyapunov) convergence rate.
	
	In most literature, the stepsize in \eqref{gd0} is required to satisfy $\alpha_k\in(0,\frac{2}{L})$ to ensure convergence. This requires knowledge of the global Lipschitz constant $L$ which is practically unknown in most cases. Moreover, recent work on large time-varying stepsize schedules has shown that sequences with steps occasionally exceeding the classical bound $\frac{2}{L}$ can still ensure stability and accelerate convergence; see, e.g., \cite{hedging}. 
	
	In this paper, we consider methods that automatically adapt the stepsize based on local smoothness. For example, \cite{adaptive_without_descent} propose
	\begin{equation}\label{adap0}
		\alpha_k=\min\left\{\sqrt{1+\frac{\alpha_{k-1}}{\alpha_{k-2}}}\alpha_{k-1},\frac{1}{2L_k}\right\}   
	\end{equation}
	wherein $\sqrt{1+\frac{\alpha_{k-1}}{\alpha_{k-2}}}\alpha_{k-1}$ controls the growth rate of $\alpha_{k}$ to ensure stability, and the term
	\begin{equation}\label{eq:Lk}
		L_k:=\frac{\|\nabla f(x_k)- \nabla f(x_{k-1})\|}{\|x_k-x_{k-1}\|}
	\end{equation}
	estimates the local smoothness of the gradient.
	
	We aim to deepen our understanding of the behavior of local estimates $L_k$, and possibly use it to obtain optimal stepsizes.
	In order to achieve analytical insight, we focus on quadratic cost functions.
	We note that this case is relevant to smooth convex optimization,
	too, because the cost function is approximately quadratic near the points of minimum. 
	
	The general quadratic optimization problem is given by 
	\begin{equation} \label{problem1}
		\min_{x\in\mathbb{R}^n}\quad f(x):=\frac{1}{2}x^{\top}Hx+b^{\top}x+c
	\end{equation}
	where $H\in\mathbb{R}^{n\times n}$ is a symmetric positive definite matrix, $b\in\mathbb{R}^n$, and $c\in\mathbb{R}$. 
	Without loss of generality, we make the change of variable $$\xi:=Q(x-x^*)$$ where $Q\in\mathbb{R}^{n\times n}$ and $x^*=-H^{-1}b$, and consider the equivalent concise problem
	\begin{equation}\label{problem2}
		\min_{\xi\in\mathbb{R}^n} \quad f(\xi):=\frac{1}{2}\xi^{\top}\Lambda\xi,
	\end{equation}
	where $\Lambda=Q^{\top} H Q$ is diagonal and we write it as
	\begin{equation}\label{4_0}
		\Lambda=
		\begin{bmatrix}
			\lambda_1 &        &        \\
			& \ddots &        \\
			&        & \lambda_n
		\end{bmatrix}.
	\end{equation}
	The terms $\{\lambda_i\}_{i=1}^n$ are the eigenvalues of $H$.
	Without loss of generality,
	we sort them in non-descending order $\lambda_1\leq\lambda_2\cdots\leq\lambda_n$ and denote minimal and maximal eigenvalues respectively as $\mu:= \lambda_1$ and $L:=\lambda_n$. The gradient of cost function $f(\xi)$ is $\nabla f(\xi_k)=\Lambda \xi_k$ and the recursive sequence $\xi_k$ generated by gradient descent reads
	\begin{equation}\label{gd}
		\begin{aligned}
			\xi_{k+1}&=\xi_k-\alpha_k\nabla f(\xi_k)\\
			&=(I-\alpha_k\Lambda)\xi_k,
		\end{aligned}
	\end{equation}
	wherein the $i$-th element of $\xi_{k}$ is computed as $\xi_{k}^i=\left( \prod_{t=0}^{k-1}(1-\alpha_t\lambda_i) \right)\xi_0^i$ for $i=1,2,...,n$.
	
	Following this line, our motivation is to utilize the local curvature estimates for problem \eqref{problem2}
	\begin{equation}\label{L_k}
		L_k=\frac{\|\nabla f(\xi_k)- \nabla f(\xi_{k-1})\|}{\|\xi_k-\xi_{k-1}\|}
		=\frac{\|\Lambda(\xi_k-\xi_{k-1})\|}{\|\xi_k-\xi_{k-1}\|}
	\end{equation}
	and design an adaptive stepsize law to achieve the optimal constant stepsize. Note that~\eqref{4_0} and~\eqref{L_k} imply
	\begin{equation}\label{9}
		\mu\leq L_k\leq L, \quad\forall \xi_{k-1},\xi_k.
	\end{equation}
	
	\section{Preliminaries}\label{sec:preliminaries}
		
	Let us first consider a constant stepsize $\alpha_k\equiv\alpha$.
	Since this determines the convergence of~\eqref{gd0}, we write the Lyapunov convergence rate as a function of stepsize, i.e., $\rho(\alpha)$.
	We characterize $\rho(\alpha)$ in the following proposition.
	\begin{prop}\label{prop1}
		The Lyapunov convergence rate of gradient descent~\eqref{gd0} for problem~\eqref{problem2} with stepsize $\alpha>0$ is given by
		\begin{equation}
			\rho(\alpha)=
			\left\{
			\begin{aligned}
				1-\mu\alpha, \quad & \text{if} \,\, \alpha\in\left(0,\frac{2}{L+\mu}\right)\\
				\frac{L-\mu}{L+\mu},\quad     &\text{if} \,\, \alpha=\frac{2}{L+\mu}\\
				L\alpha-1, \quad & \text{if}\,\, \alpha>\frac{2}{L+\mu}.
			\end{aligned}
			\right.
		\end{equation}
	\end{prop}
	
	\begin{pf}
		In light of~\eqref{gd},
		the claim follows by computing the slowest mode (eigenvalue) as $\max\{|1-\alpha \mu|,|1-\alpha L|\}$. \qed
	\end{pf}
    
    \begin{figure}[t]
		\centering
		{\includegraphics[scale=0.4]{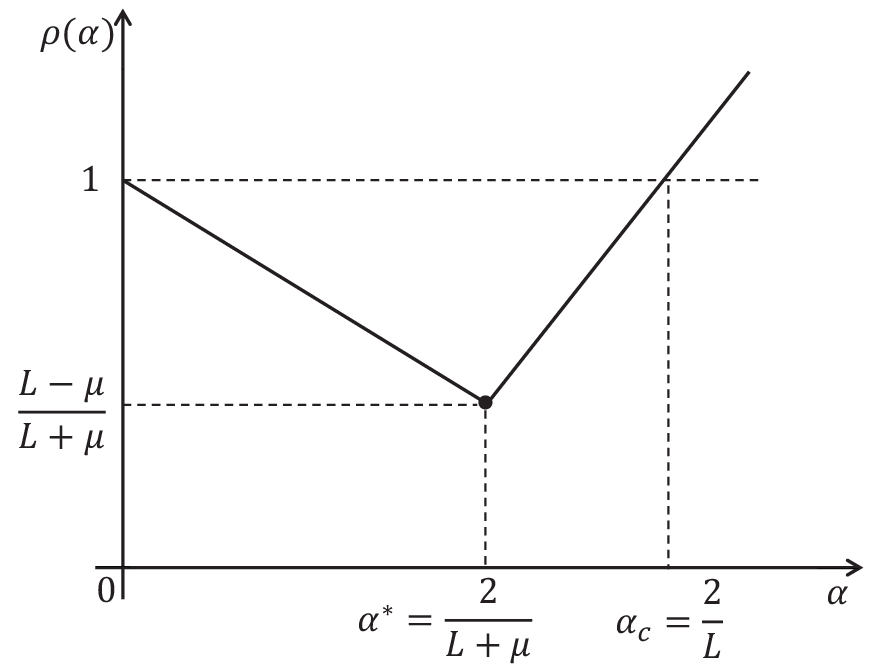}}
		\caption{Illustration of Lyapunov convergence rate $\rho(\alpha)$.}\label{lyapunov_comp}
	\end{figure}
    
	A graphic illustration of Proposition \eqref{prop1} is given in Fig.~\ref{lyapunov_comp}. 
	We denote the optimal (fastest) constant stepsize by 
	\begin{equation}\label{eq:alpha-star}
		\alpha^*=\frac{2}{L+\mu}
	\end{equation}
	and the critical one for stability,
    that ensures $\rho(\alpha) < 1$, by
	\begin{equation}\label{eq:alpha-c}
		\alpha_c=\frac{2}{L}.
	\end{equation}
	In particular, $\xi\not\to0$ if $\alpha\ge\alpha_c$; see Fig.~\ref{lyapunov_comp}.
	We now turn to local smoothness estimate~\eqref{eq:Lk} 	which we denote by $L_k(\alpha)$ with constant stepsize $\alpha$.
	The following result holds.
	\begin{prop}\label{prop2}
		If $\xi_{0}^i\neq 0$ for all $i=1,2,...,n$, then
		\begin{equation}
			L(\alpha):=\lim_{k\rightarrow{+\infty}}L_k(\alpha)=
			\left\{
			\begin{aligned}
				\mu, & \quad \text{if}\,\, \alpha\in(0,\alpha^*)\\
				L, & \quad \text{if}\,\, \alpha>\alpha^*.
			\end{aligned}
			\right.
		\end{equation}
	\end{prop}
	
	\begin{pf}
		Recall $\mu\leq L_k\leq L$ by \eqref{9}.
		Expanding~\eqref{L_k} yields
		\begin{equation}\label{Lk_expand1}
			\begin{aligned}
				\hspace{-10pt} L_k(\alpha)&=\sqrt{ \frac{\|-\alpha\Lambda^2\xi_{k-1}\|^2}{\|-\alpha\Lambda \xi_{k-1}\|^2} }\\
				&= \sqrt{ \frac{\sum_{i=1}^{n}\lambda_i^4(1-\alpha\lambda_i)^{2(k-2)}\xi_{i,0}^2}{\sum_{i=1}^{n}\lambda_i^2(1-\alpha\lambda_i)^{2(k-2)}\xi_{i,0}^2} }.
			\end{aligned}
		\end{equation}
		According to Proposition~\ref{prop1}, the term $(1-m\alpha)^{2(k-2)}$ is the largest across all $\lambda_i$'s if $\alpha<\alpha^*$.
		If the minimal eigenvalue is unique, i.e., $\mu=\lambda_1<\lambda_2$, we have
		\begin{equation}\label{14}
			\lim_{k\rightarrow{{+\infty}}} \frac{ \lambda_i^2\xi_{i,0}^2(1-\alpha\lambda_i)^{2(k-2)} }{\lambda_1^2\xi_{1,0}^2(1-\alpha\lambda_1)^{2(k-2)}}=0,\quad \forall i\neq 1
		\end{equation}
		and plugging~\eqref{14} into~\eqref{Lk_expand1} at the limit yields
		\begin{equation}
			\lim_{k\rightarrow{{+\infty}}} L_k(\alpha)=\mu.
		\end{equation}
		If the minimal eigenvalue is repeated, i.e., $\exists l\geq 2$ such that $\mu=\lambda_1=\cdots=\lambda_l$, then it holds
		\begin{equation}\label{16}
			\lim_{k\rightarrow{{+\infty}}} \frac{ \lambda_i^2(1-\alpha\lambda_i)^{2(k-2)}\xi_{i,0}^2 }{\lambda_1^2(1-\alpha\lambda_1)^{2(k-2)}(\xi_{1,0}^2+\cdots+\xi_{l,0}^2)}=0
		\end{equation}
		for all $i>l$. 
		Subbing \eqref{16} into \eqref{Lk_expand1}, we similarly conclude that
		\begin{equation}
			\lim_{k\rightarrow{{+\infty}}} L_k(\alpha)=\lambda_1=\mu.
		\end{equation}
		If $\alpha>\alpha^*$, the dominant term in~\eqref{Lk_expand1} is $(1-L\alpha)^{2(k-2)}$ across all $\lambda_i$'s.
		Irrespectively of the algebraic multiplicity of $\lambda_i=L$, it holds
		\begin{equation}
			\lim_{k\rightarrow{{+\infty}}} L_k(\alpha)=\lambda_n=L.
		\end{equation} 
		This completes the proof. \qed
	\end{pf}
	
	\begin{figure}[t]
		\centering
		{\includegraphics[scale=0.4]{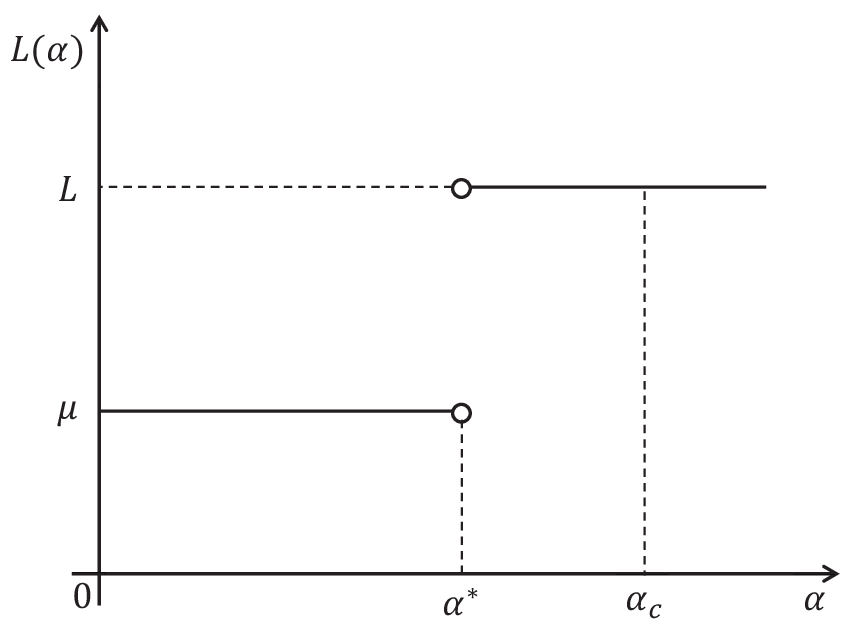}}
		\caption{$L(\alpha)$ under gradient descent with a constant stepsize $\alpha$.}\label{L_k.eps}
	\end{figure}
	
	The following lemma generalizes the previous result under time-varying stepsize $\alpha_k$.
	
	\begin{lem}\label{lem4}
		If there exists $\overline{k}\in\mathbb{N}_{+}$ and $\epsilon> 0$ such that $\xi_{\overline{k}}^i\neq 0$ for $i=1,...,n$ and $\alpha_k\leq \alpha^*-\epsilon$ for all $k>\overline{k}$, then $\lim_{k\rightarrow{+\infty}}L_k(\alpha_k)=\mu$. If $\alpha_k\geq \alpha^*+\epsilon$ for all $k>\overline{k}$,
		then $\lim_{k\rightarrow{+\infty}}L_k(\alpha_k)=L$.
	\end{lem}
	
	\begin{pf}
		We consider the two cases separately.
		\begin{description}
			\item[Case $\alpha_k\leq \alpha^*-\epsilon$.] Let us rewrite $\xi_k^i$ as 
			\begin{equation}\label{19_0}
				\xi_{k}^i=\prod_{t=\bar{k}}^{k-1}(1-\alpha_t\lambda_i)\xi_{\bar{k}}^i.
			\end{equation}
			Define $r_{k}^i:=\prod_{t=\bar{k}}^{k-1}(1-\alpha_t\lambda_i)$ and $r_k^i=r_k^1$ for $i\in\{1,...,l\}$ where $\lambda_1=\cdots=\lambda_l=\mu$ and $\lambda_{l+1}>\mu$.
			As $\frac{|1-\alpha_k\lambda_i|}{|1-\mu\alpha_k|}<1$ for all $\lambda_i>\mu$ and $k\geq\bar{k}$, it follows
			\begin{equation}\label{20_0}
				\lim_{k\rightarrow {+\infty}}\frac{r_k^i}{r_k^1}=
				\left\{
				\begin{aligned}
					1, &\quad \text{if}\,\, i\in\{1,...,l\}\\
					0, &\quad \text{otherwise}.
				\end{aligned}
				\right.
			\end{equation}
			Combining \eqref{19_0} with \eqref{20_0} yields
			\begin{equation}\label{21_0}
				\begin{aligned}
					\lim_{k\rightarrow{+\infty}}\frac{\xi_k}{r_k^1}\hspace{-1pt}&=\hspace{-1pt}\lim_{k\rightarrow {+\infty}}\left[\xi_{\bar{k}}^1,\xi_{\bar{k}}^2,...,\xi_{\bar{k}}^l,\frac{r_k^{l+1}}{r_k^1} \xi_{\bar{k}}^{l+1},...,\frac{r_k^{n}}{r_k^1} \xi_{\bar{k}}^{n}\right]^{\top}\\
					\hspace{-1pt}&=\hspace{-1pt}\lim_{k\rightarrow {+\infty}}[\xi_{\bar{k}}^1,\xi_{\bar{k}}^2,...,\xi_{\bar{k}}^l,0,...,0]^{\top},
				\end{aligned}
			\end{equation}
			which implies 
			\begin{eqnarray}\label{22_0}
				\lim_{k\rightarrow{+\infty}}\frac{\Lambda \xi_k}{r_k^1}&=& \lim_{k\rightarrow{+\infty}} \mu [\xi_{\bar{k}}^1,\xi_{\bar{k}}^2,...,\xi_{\bar{k}}^l,0,...,0]^{\top}
			\end{eqnarray}
			and 
			\begin{equation}\label{23_0}
				\lim_{k\rightarrow{+\infty}}\frac{\Lambda^2\xi_k}{r_k^1}=\lim_{k\rightarrow{+\infty}}\mu^2[\xi_{\bar{k}}^1,\xi_{\bar{k}}^2,...,\xi_{\bar{k}}^l,0,...,0]^{\top}.
			\end{equation}
			In light of \eqref{L_k}, \eqref{22_0} and \eqref{23_0}, we obtain
			\begin{eqnarray}\label{24_0}
				\lim_{k\rightarrow {+\infty}}L_k &=& \lim_{k\rightarrow {+\infty}} \frac{\|\Lambda(\xi_k-\xi_{k-1})\|}{\|\xi_k-\xi_{k-1}\|}\nonumber\\
				&=& \lim_{k\rightarrow {+\infty}}\frac{\|-\frac{\alpha_{k-1}\Lambda^2\xi_{k-1}}{r_k^1}\|}{\|-\frac{\alpha_{k-1}\Lambda\xi_{k-1}}{r_k^1}\|}=\mu.
			\end{eqnarray}
			\item[Case $\alpha_k\geq \alpha^*+\epsilon$.] Note that $r_k^i=r_k^n$ for all $i\in\{q,...,n\}$ such that $\lambda_i=L$ for some $q\ge l$. 
			Similar to the derivation from \eqref{19_0} to \eqref{24_0},   
			we deduce that
			\begin{eqnarray}
				\lim_{k\rightarrow {+\infty}}L_k =L.
			\end{eqnarray}
		\end{description} 
		The proof is completed.\qed
	\end{pf}
	
	Based on the above preliminaries, one question arises: do we get convergence under an adaptive stepsize law $\alpha_k=\mathcal{F}(L_k,\alpha_{k-1})$ where $\mathcal{F}$ is an update law to be chosen?
    The term $L_k$ seems useful to \emph{estimate} minimal and maximal local curvatures via estimators $\hat{\mu}_k:=\min_{k}\{L_k\}$ and $\hat{L}_k:=\max_k\{L_k\}$ which can be used to design an adaptive stepsize inspired by~\eqref{eq:alpha-star}--\eqref{eq:alpha-c}.
	The choice of $\alpha_k$ is driven by two objectives: (\textit{i}) to stabilize the system; (\textit{ii}) to explore the landscape to get good smoothness estimates.
	In the next section, we present our algorithm to achieve the optimal constant stepsize~\eqref{eq:alpha-star} without knowing $\mu$ and $L$ \emph{a priori}.

    \begin{rem}
        We purposely do not use parameters $L$ and $\mu$ in our design since, although they are known for the simple problem~\eqref{problem2}, their computation for general convex functions---which we aim to target in future work---is intractable.
    \end{rem}
    
	\begin{rem}
		In Proposition \ref{prop2}, we neglect the case $\alpha=\alpha^*$ because it does not affect our convergence analysis.
	\end{rem}
	
	\section{Main Results}\label{sec:main-results}
	Since establishing stability contributes to the convergence analysis, we first offer intuition for the adaptive stepsize
	\begin{equation}\label{adap1}
		\alpha_k=\frac{\beta}{\hat{L}_k},
	\end{equation}
	where $\beta$ is a constant design parameter and 
	\begin{equation}\label{eq:Lmax}
		\hat{L}_k=\max\{\hat{L}_{k-1},L_k\}, \quad \hat{L}_1=L_1.
	\end{equation}
	Note that $\hat{L}_k$ is non-decreasing and $\hat{L}_k\le L$ in view of~\eqref{4_0} and~\eqref{L_k}. By monotonicity and boundedness, the limit $\overline{L}:=\lim_{k\rightarrow{+\infty}}\hat{L}_k$ exists finite.
	
	In the sequel, we provide necessary and sufficient conditions for stability under the adaptive stepsize law~\eqref{adap1}.
	\begin{prop}\label{prop3}
		Let $\alpha_0>0$ and $\xi_{0}^i\neq 0$ for all $i=1,2,...,n$,
		then $\xi_k\rightarrow 0$ if and only if $\beta\in(0,2)$.
	\end{prop}
	\begin{pf}
		We prove it by contradiction in two steps.
		\begin{description}
			\item[Sufficiency ($\beta\in(0,2)\implies\xi_k\rightarrow{0}$).] Suppose that $\xi_k\nrightarrow 0$ if $\beta\in(0,2)$. We first claim that $\xi_k\nrightarrow 0$ implies that there exists $\overline{k}\in\mathbb{N}_+$ such that $\alpha_k\geq \alpha_c$ for all $k>\overline{k}$. If this is not true, then there exists $\overline{k}\in\mathbb{N}_+$ and $k>\overline{k}$ such that $\alpha_k<\alpha_c$. Noting that $\alpha_k$ is monotonically non-increasing, then the opposite statement becomes there exist $\epsilon>0$ and $\overline{k}\in\mathbb{N}_+$ such that $\alpha_k\leq\alpha_c-\epsilon$ for all $k>\overline{k}$. It follows that
			\begin{eqnarray}
				\lim_{k\rightarrow{+\infty}}\xi_k^i&=&\lim_{k\rightarrow{+\infty}}\prod_{t=\overline{k}}^{k-1}(1-\alpha_t\lambda_i)\xi_{\overline{k}}^i=0,
			\end{eqnarray}
			because $|1-\alpha_t\lambda_i|<1$ for all $i=1,...,n$, and $|\xi_{\overline{k}}^i|<+\infty$. This observation contradicts $\xi_k\nrightarrow0$, and thus, the statement that $\xi_k\nrightarrow 0$ means $\alpha_k\geq \alpha_c$ for all $k>\overline{k}$ and $\overline{k}\in\mathbb{N}_+$ is true. 
			Since $\alpha_c=\frac{2}{L}>\frac{2}{L+\mu}=\alpha^*$, there exists an $\epsilon>0$ such that $\alpha_c\geq \alpha^*+\epsilon$. 
			According to Lemma~\ref{lem4}, $\alpha_k\geq \alpha_c$ further implies that $\lim_{k\rightarrow {+\infty}}L_k=L$. By $\hat{L}_k:=\max\{\hat{L}_{k-1},L_k\}$, we have
			\begin{equation}
				\lim_{k\rightarrow{+\infty}}\hat{L}_k=\overline{L}=L
			\end{equation}
			and the stepsize updates \eqref{adap1} have limit
			\begin{equation}
				\lim_{k\rightarrow {+\infty}}\alpha_k= \lim_{k\rightarrow {+\infty}}\frac{\beta}{\hat{L}_k}=\frac{\beta}{L}<\alpha_c,
			\end{equation}
			which contradicts the hypothesis $\alpha_k\geq \alpha_c$ for all $k>\overline{k}$ if $\beta\in(0,2)$.
			\item[Necessity ($\xi_k\rightarrow{0}\implies\beta\in(0,2)$).] Suppose that there exists $\beta\geq 2$ such that $\xi_k\rightarrow{0}$ for all $\xi_0^i\neq 0$.
			Let us consider a scalar system with $\mu=L=1$ and initial conditions $\xi_0=1$, $\alpha_0\neq 1$.
			Then, the iterates $\xi_k$ read
			\begin{equation}
				\xi_k=(1-\beta)\xi_{k-1}=(1-\beta)^{k-1}(1-\alpha_0),
			\end{equation}
			which implies that $\xi_k$ either oscillates if $\beta=2$ or diverges if $\beta>2$, contradicting the hypothesis.
		\end{description}
		This completes the proof. \qed
	\end{pf}
	
	Proposition~\ref{prop3} answers the first question about stability requirement under an adaptive stepsize law like \eqref{adap1} and offers an intuitive direction to utilize the non-decreasing monotonicity of $\overline{L}_k$ for stabilizing stepsizes.
	
	Now, we are in the position to explore an adaptive rule that converges to the optimal constant stepsize $\alpha^*$ in~\eqref{eq:alpha-star}.
	Let us introduce the minimal local curvature estimator
	\begin{equation}\label{eq:Lmin}
		\hat{\mu}_k=\min_{k}\{L_k\} = \min\{\hat{\mu}_{k-1},L_k\}, \quad \hat{\mu}_1 =L_1.
	\end{equation}
	Since it is non-increasing and lower-bounded by $\mu$,
    the limit $\overline{\mu}:=\lim_{k\rightarrow {+\infty}}\hat{\mu}_k$ exists finite analogous to $\overline{L}$.
	We introduce the following stepsize which,
	together with~\eqref{gd}, we name adaptive optimal gradient descent (AdOGD):
	\begin{equation}\label{adap2}
		\alpha_k=\frac{2}{\hat{\mu}_k+\hat{L}_k}.
	\end{equation}
	The limit of $\alpha_k$ exists because $\hat{\mu}_k$ and $\hat{L}_k$ converge to their limits $\overline{\mu}$ and $\overline{L}$, respectively.
	Also,
	it holds $$\mu\leq\overline{\mu}\leq \hat{\mu}_k\leq\hat{L}_k\leq\overline{L}\leq L.$$
	
	We next present our main convergence result for AdOGD.
	
	\begin{thm}\label{thm1}
		Under the adaptive stepsize law AdOGD, if $\xi_{0}^i\neq 0$ for $i=1,2,...,n$, then $\hat{\alpha}:=\lim_{k\rightarrow{+\infty}}\alpha_{k}=\alpha^*$.
	\end{thm}
	\begin{pf}
		The proof is by contradiction. 
		Suppose $\hat{\alpha}\neq\alpha^*$, i.e., either $\hat{\alpha}\in(0,\alpha^*)$ or $\hat{\alpha}>\alpha^*$. 
		\begin{description}
			\item[Case $\hat{\alpha}\in(0,\alpha^*)$.] By hypothesis,
			$\exists\overline{k}\in\mathbb{N}_+$ such that
			\begin{equation}\label{34}
				\alpha_k\leq\frac{\alpha^*+\hat{\alpha}}{2}<\alpha^*, \quad\forall k>\overline{k}.
			\end{equation}
			Equivalently, we rewrite \eqref{34} as 
			\begin{equation}
				\alpha_k\leq\alpha^*-\frac{\alpha^*-\hat{\alpha}}{2}\leq\alpha^*-\epsilon, \quad\forall k>\overline{k},
			\end{equation}
			for some positive constant $\epsilon\leq \frac{\alpha^*-\hat{\alpha}}{2}$.
			According to Lemma~\ref{lem4}, it holds $\lim_{k\rightarrow{+\infty}}L_k=\mu$ and, by rule~\eqref{eq:Lmin},
			\begin{equation}
				\lim_{k\rightarrow{+\infty}}\hat{\mu}_k=\mu.
			\end{equation}
			It follows that
			\begin{equation}
				\lim_{k\rightarrow {+\infty}} \alpha_k=\frac{2}{\mu+\overline{L}}\geq \frac{2}{\mu+L}= \alpha^*
			\end{equation}
			which contradicts the hypothesis $\hat{\alpha}\in(0,\alpha^*)$.
			\item[Case $\hat{\alpha}>\alpha^*$.] Similarly to the previous case,
			by using Lemma~\ref{lem4} and monotonicity of rule~\eqref{eq:Lmax}, we derive $\lim_{k\rightarrow{+\infty}}L_k=L$ and $
				\lim_{k\rightarrow{+\infty}}\hat{L}_k=L$,
			which yields
			\begin{equation}
				\lim_{k\rightarrow {+\infty}} \alpha_k=\frac{2}{\overline{\mu}+L}\leq \frac{2}{\mu+L}= \alpha^*.
			\end{equation}
			This contradicts the hypothesis $\hat{\alpha}>\alpha^*$.
		\end{description}
		The previous two cases prove that $\alpha_k$ under the adaptive stepsize law AdOGD converges to $\alpha^*$. \qed
	\end{pf}
	
	\section{Numerical Experiments}\label{sec:numerical-study}
	We compare the performance of our proposed approach with some state-of-the-art adaptive algorithms.
	Consider the quadratic optimization problem~\eqref{problem1} where $x\in\mathbb{R}^5$ and the eigenvalues of $\Lambda$ are linearly distributed between $\mu=1$ and $L=9$. 
	We compare our adaptive stepsize~\eqref{adap2} with three recent algorithms:
	AdGD by \cite{adaptive_without_descent} shown in \eqref{adap0};
	AFFGD by \cite{iannelli2025};
	AdaGM by \cite{adgd_proximal2}.
    Simulation results are reported in Fig.~\ref{integrated_fig3}.

    \begin{figure}[htbp]
        \centering
        \subfloat[Stepsizes with the compared adaptive strategies.]{
            \begin{minipage}[h]{0.95\linewidth}
                \centering
                \includegraphics[width=.8\textwidth]{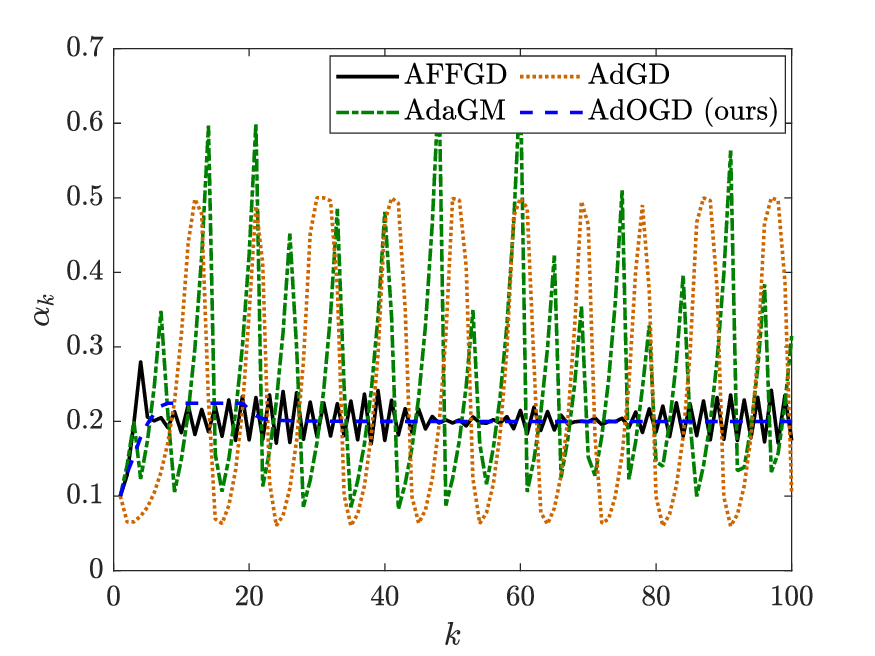}
                \label{fig4}
            \end{minipage}
        }\\ \vspace{-8pt}
        \subfloat[Our proposed stepsize converges to $\alpha^*$.]{
            \begin{minipage}[h]{0.95\linewidth}
                \centering
                \includegraphics[width=.8\textwidth]{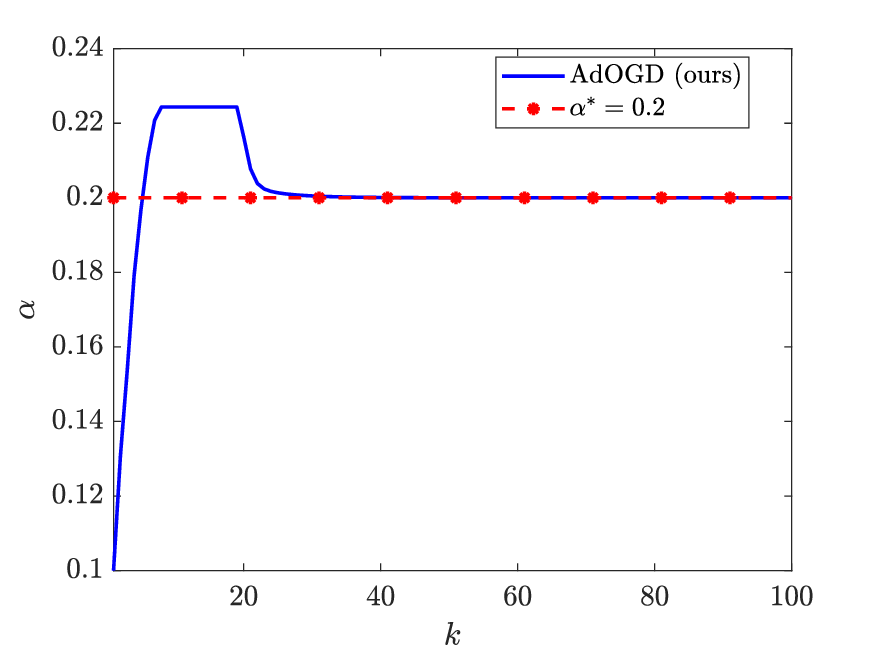}
                \label{fig4_added}
            \end{minipage}
        }\\ \vspace{-8pt}
        \subfloat[Estimates of local curvature with AdOGD.]{
            \begin{minipage}[h]{0.95\linewidth}
                \centering
                \includegraphics[width=.8\textwidth]{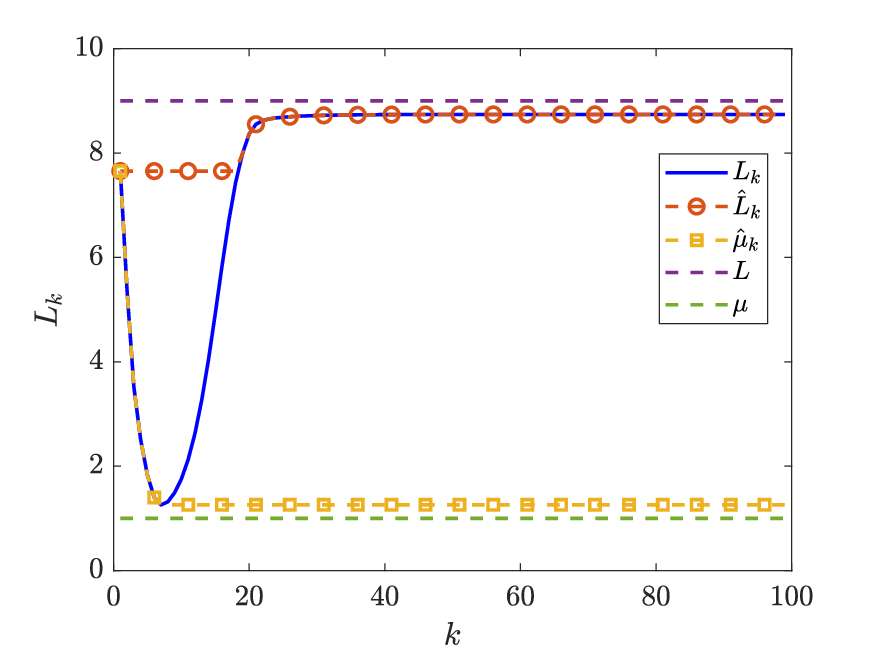}
                \label{fig3}
            \end{minipage}
        }\\  \vspace{-8pt}
        \subfloat[Lyapunov convergence rate.]{
            \begin{minipage}[h]{0.95\linewidth}
                \centering
                \includegraphics[width=.8\textwidth]{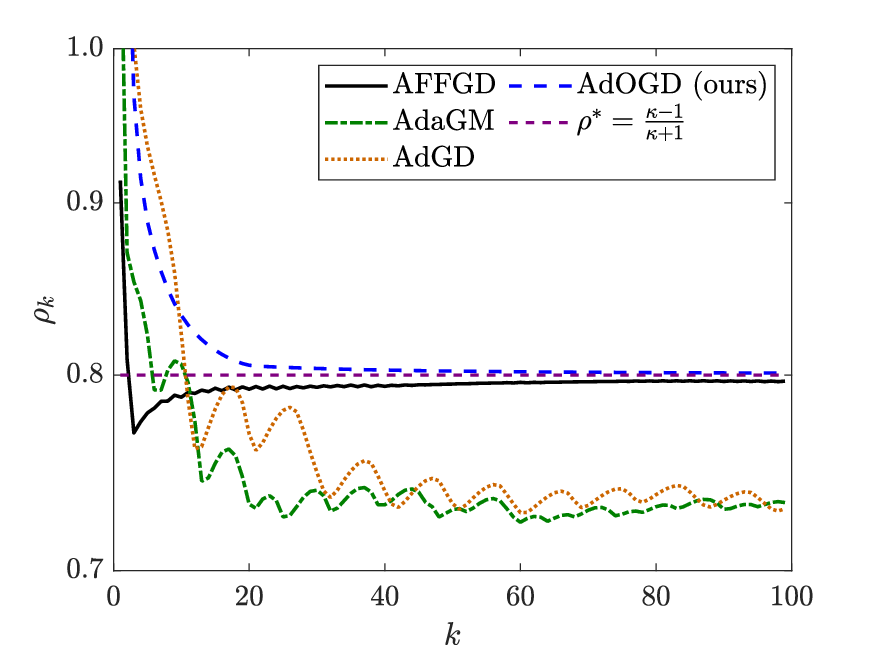}
                \label{fig5}
            \end{minipage}
        }
        \caption{Comparison with state-of-the-art algorithms for quadratic cost with $\mu=1$ and $L=9$.}
        \label{integrated_fig3}
    \end{figure}
	
	Figure~\ref{fig4} plots the trajectories of adaptive stepsizes.
	These significantly oscillate with AdaGM and AdGD while AFFGD exhibits modest oscillations.
	By contrast,
	the stepsize of our proposed AdOGD algorithm converges to the optimal constant value according to Theorem~\ref{thm1}; see Fig.~\ref{fig4_added}.
	Notably,
    the local smoothness estimates do not converge to $\hat{\mu}_k\not\to\mu$ and $\hat{L}_k\not\to L$,
	as shown by Fig.~\ref{fig3}.
	
	\begin{figure}
		\centering
		\includegraphics[width=.36\textwidth]{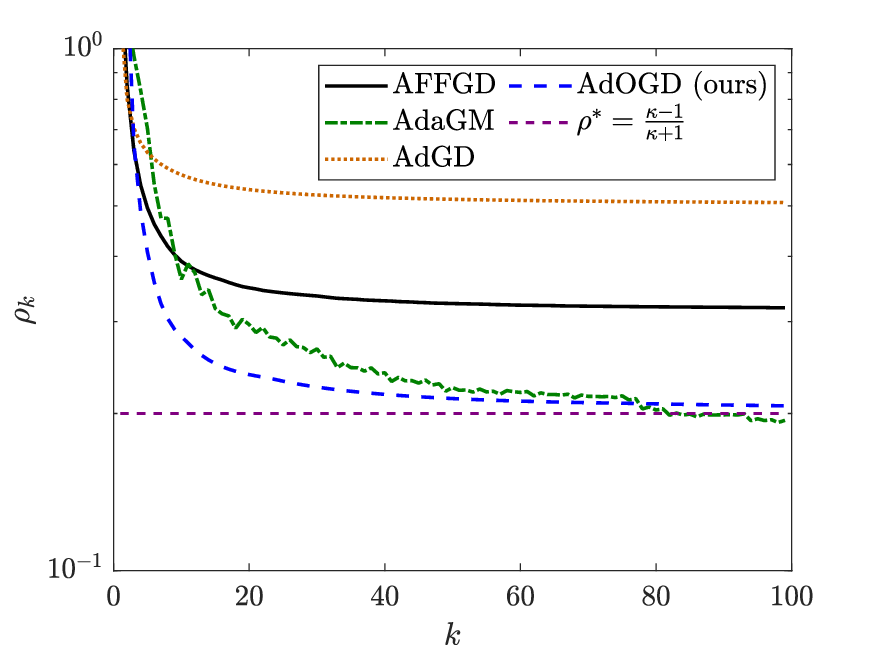}
		\caption{Lyapunov convergence rate for quadratic cost with $\mu=1$ and $L=1.5$.}\label{fig6}
	\end{figure}
	Figure~\ref{fig5} shows the Lyapunov convergence rate $\rho_k:=\exp(\frac{\ln\|x_k-x^*\|}{k})$.
	The asymptotic rate $\rho=\lim_{k\rightarrow{{+\infty}}}\rho_k$ characterizes the speed of algorithms in the long run and $\rho^*:=\frac{\kappa-1}{\kappa+1}$ captures the Lyapunov convergence rate with optimal constant stepsize $\alpha^*$, where $\kappa=\frac{L}{\mu}$ is the condition number. 
	Both our proposed AdOGD and AFFGD ultimately achieve the optimal rate with constant stepsize $\rho^*$ while AdaGM and AdGD are much faster.
    On the other hand, reducing the largest eigenvalue to $L=\kappa=1.5$, our algorithm is outperformed only by AdaGM and $\rho^*$; see Fig.~\ref{fig6}. This suggests that the relative performance of the compared methods depends on the condition number of the problem: when $\kappa$ is smaller, the optimal constant-step benchmark is already faster, leaving less room for oscillating stepsizes to provide an additional advantage.
	
	
	To get a sense of the performance of our approach in a more general setting,
	we preliminarily test with a strongly convex function for binary classification.
	Given $N$ features $s_i\in\mathbb{R}^n$ and labels $y_i\in\{\pm 1\}$, the goal is to find a linear classifier $x\in\mathbb{R}^n$ by solving 
	\begin{equation}\label{41}
		\min_{x \in \mathbb{R}^n}\;
		\frac{1}{N} \sum_{i=1}^N \log\!\left(1 + \exp(-y_i x^\top s_i)\right)
		+ \frac{\mu}{2} \|x\|^2.
	\end{equation}
	In this case,
	the smoothness constant is $L=\frac{1}{4N}\lambda_{\max}(S)^2$ where $S\in\mathbb{R}^{N\times n}$ is the feature matrix
	and the strong convexity constant is $\mu=0.01$. 
	To benchmark adaptive algorithms against an equivalent of the optimal constant stepsize in quadratic optimization,
    which has no standard definition with generic convex costs,
    we define the iteration–dependent critical and optimal stepsizes using the eigenvalues of the Hessian matrix evaluated at $x_k$
    \begin{equation}\label{alpha_c}
        \alpha_k^c:=\frac{2}{\lambda_{\max}(\nabla^2 f(x_k))}
    \end{equation}
    \begin{equation}\label{alpha_star}
        \alpha_k^*:=\frac{2}{\lambda_{\min}(\nabla^2f(x_k))+\lambda_{\max}(\nabla^2 f(x_k))}.
    \end{equation}
	Remarkably,
    Fig.~\ref{fig7} shows that $\alpha_k^\star$ and the adaptive stepsize $\alpha_k$ of AdOGD converge to (about the) same value,
    offering ground for systematic extension of our algorithm to convex optimization.
    Figure~\ref{fig9} shows that AdOGD and the benchmarks achieve the same convergence rate.
    This fact yields an interesting insight.
    Although adapting the stepsize online is practically necessary to ensure stability, making it oscillate indefinitely need not be the most effective strategy to speed up convergence in some cases.
    In this sense, our algorithm might prove a simple but effective choice compared to the benchmarks.
    
    \begin{figure}
        \centering
        \subfloat[Stepsizes with the compared adaptive strategies.]{
            \begin{minipage}[h]{0.95\linewidth}
                \centering
                \includegraphics[width=.76\textwidth]{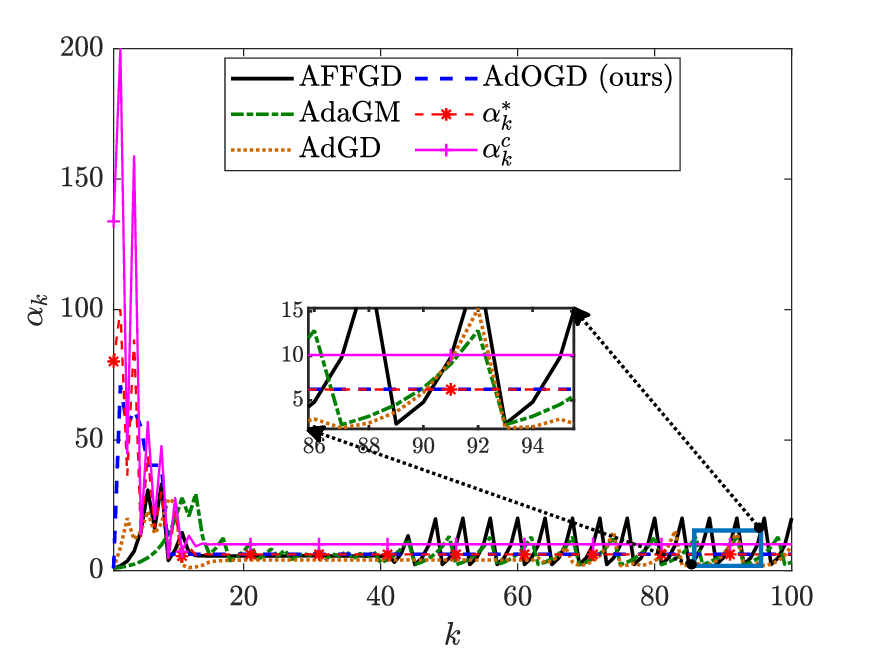}
    		    \label{fig7}
            \end{minipage}
        }\\ 
        \subfloat[Lyapunov convergence rate.]{
            \begin{minipage}[h]{0.95\linewidth}
                \centering
    		\includegraphics[width=.76\textwidth]{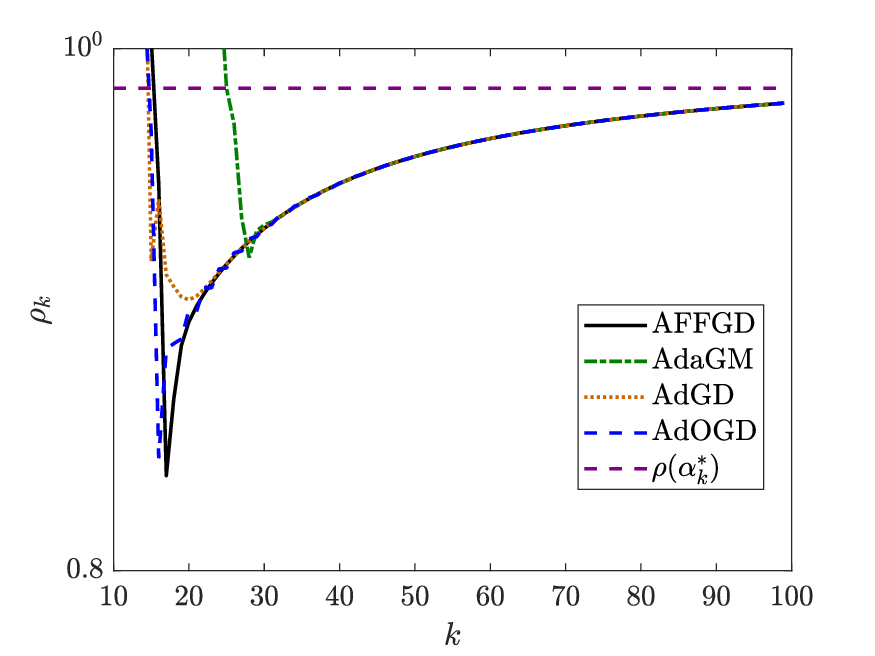}
    		\label{fig9}
            \end{minipage}
        }
        \caption{Comparison with state-of-the-art algorithms for the strongly convex function         in~\eqref{41}.}
        \label{integrated_fig5}
    \end{figure}
    \section{Conclusion}\label{sec:conclusion}
	We have designed an adaptive stepsize for gradient descent to achieve the fastest asymptotic convergence with constant stepsize.
	We have used estimates of minimal and maximal local curvatures to track the optimal constant stepsize for quadratic optimization problems.
	Theoretical results prove asymptotic convergence of our proposed stepsize,
	and numerical studies support that our approach effectively speeds up the convergence.
	Moreover, we have tested with a strongly convex function where the convergence rate is as fast as the benchmarks.
	Future research will broaden the scope of the present study to fastest gradient descent for convex functions. 
	To this aim,
	we point out that results discussed by \cite{hedging} provide promising ground for varying stepsizes.
    In addition,
    we aim to extend our adaptive strategy to distributed settings where global parameters are unknown.

\end{document}